\documentclass[11pt]{article}
\usepackage{color,latexsym,amsthm,amsmath,amssymb,path,enumitem,url}

\newcommand{\ignore}[1]{}

\newtheorem{theorem}{Theorem}[section]
\newtheorem{lemma}[theorem]{Lemma}
\newtheorem{corollary}[theorem]{Corollary}

\newtheorem{conjecture}[theorem]{Conjecture}

\newcommand{\Proof}[1]
        {
        \noindent
        \emph{Proof #1.}~
        }
\newsavebox{\smallProofsym}                     
\savebox{\smallProofsym}                        %
        {
        \begin{picture}(7,7)                    %
        \put(0,0){\framebox(6,6){}}             %
        \put(0,2){\framebox(4,4){}}             %
        \end{picture}                           %
        }                                       %
\newcommand{\smalleop}[1]
        {
        \mbox{} \hfill #1~~\usebox{\smallProofsym}\!\!\!\!\!\!\
        }

\newcommand{\parag}[1]{\vspace{2mm}

\noindent{\bf #1} }

\usepackage{graphicx,psfrag}
\usepackage[rflt]{floatflt}

\usepackage{dsfont}

\newcommand{\ZZ}{\ensuremath{\mathbb Z}}
\newcommand{\RR}{\ensuremath{\mathbb R}}

\newcommand{\pts}{\mathcal P}
\newcommand{\tts}{\mathcal T}
\newcommand{\curves}{\Gamma}
\newcommand{\circs}{\mathcal{C}}

\newcommand{\vb}{{\bf V}}

\DeclareMathOperator*{\EE}{\mathbb{E}}

\def\eps{{\varepsilon}}

\begin{document}
\pagenumbering{arabic}

\title{Higher Distance Energies and Expanders with Structure}

\author{
Cosmin Pohoata\thanks{California Institute of Technology, Pasadena, CA, USA
{\sl apohoata@caltech.edu}.}
\and
Adam Sheffer\thanks{Department of Mathematics, Baruch College, City University of New York, NY, USA.
{\sl adamsh@gmail.com}. Supported by NSF grant DMS-1710305}}

\maketitle

\begin{abstract}
We adapt the idea of higher moment energies, originally used in Additive Combinatorics, so that it would apply to problems in Discrete Geometry.
This new approach leads to a variety of new results, such as \\
(i) Improved bounds for the problem of distinct distances with local properties.\\
(ii) Improved bounds for problems involving expanding polynomials in $\RR[x,y]$ (Elekes--R\'onyai type bounds) when one or two of the sets have structure.

Higher moment energies seem to be related to additional problems in Discrete Geometry, to lead to new elegant theory, and to raise new questions.
\end{abstract}

\section{Introduction}

In this work we use techniques from Additive Combinatorics to derive new results for Discrete Geometry problems.
We obtain two types of results by using similar techniques: new bounds for several distinct distances problems and new bounds for expanding polynomials when the sets have some kind of structure.

The \emph{Erd\H os distinct distances problem} is a main problem in Discrete Geometry, which asks for the
minimum number of distinct distances spanned by a set of $n$ points in $\RR^2$.
That is, denoting the distance between two points $p, q \in \RR^2$ as $|pq|$, the problem asks for $\min_{|\pts|=n} |\{|pq| :\  p, q \in \pts \}|$.
Note that $n$ equally spaced points on a line span $n-1$ distinct distances.
Erd\H os \cite{erd46} observed that a $\sqrt{n}\times \sqrt{n}$ section of the integer lattice $\ZZ^2$ spans $\Theta(n/\sqrt{\log n})$ distinct distances.
Proving that every point set determines at least some number of distinct distances turned out to be a deep and challenging problem.

The above problem is just one out of a large family of distinct distances problems, including higher-dimensional variants, structural problems, and many other types of problems (for example, see \cite{Sheffer14}).
The main problems in this family were proposed by Erd\H os and have been studied for decades.
After over 60 years and many works on distinct distances problems, Guth and Katz \cite{GK15} almost settled the original question by proving that every set of $n$ points in $\RR^2$ spans $\Omega(n/ \log n)$ distinct distances.
Surprisingly, so far this major discovery was not followed by significant progress in the other main distinct distances problems.

In recent years various results in Additive Combinatorics have been obtained by using \emph{higher moment energies}, which generalize the concept of additive energy (e.g., see \cite{Schoen16,ScSh13,SchoShk11}).
When studying distinct distances problems, one often relies on a set of quadruples that can be thought of as a variant of additive energy (for example, see \cite{Chara12,FPS17,GK15,PdZ17,SSS13}).
It seems fitting to refer to this object as the \emph{distance energy} of a point set.
We extend the idea of higher moment energies by defining the concept of \emph{higher distance energy}.
This concept is described in Section \ref{sec:HigerEner}.

The use of higher distance energy leads to several new distinct distances results, and to possible strategies for studying other such problems.
Moreover, it turns out that just thinking of distinct distances in terms of energy leads to various new observations.
A similar situation occurs for the family of problems involving expanding polynomials.
We believe that this work exposes just another part of a strong connection between the fields of Additive Combinatorics and Discrete Geometry.
Our hope is that this connection will continue to unfold, and we plan to continue pursuing this direction.

\parag{Distinct distances with local properties.}
Let $\phi(n,k,l)$ denote the minimum number of distinct distances that are determined by a planar $n$ point set $\pts$ with the property that any $k$ points of $\pts$ determine at least $l$ distinct distances. That is, by having a local property of every small subset of points, we wish to obtain a global property of the entire point set.

For example, the value of $\phi(n,3,3)$ is the minimum number of distinct distances that are determined by a set of $n$ points that do not span any isosceles triangles (including degenerate triangles with three collinear vertices).
Since no isosceles triangles are allowed, every point determines $n-1$ distinct distances with the other points of the set, and we thus have $\phi(n,3,3) = \Omega(n)$.
Erd\H os \cite{Erdos86} observed the following upper bound for $\phi(n,3,3)$.
Behrend \cite{Behrend46} proved that there exists a set $A$ of positive integers $a_1< a_2 < \cdots < a_n$ such that no three elements of $A$ determine an arithmetic progression and $a_n < n2^{O(\sqrt{\log n})}$.
Therefore, the point set $\pts_1 = \{(a_1,0), (a_2,0),\ldots, (a_n,0)\}$ does not span any isosceles triangles.
Since $\pts_1 \subset \pts_2 = \{(1,0),(2,0),\ldots,(a_n,0) \}$ and $D(\pts_2)< n2^{O(\sqrt{\log n})}$, we have $\phi(n,3,3) < n2^{O(\sqrt{\log n})}$.

For any $k\ge 4$ we have
\[ \phi\left(n,k,\binom{k}{2}-\lfloor k/2 \rfloor +2 \right) = \Omega\left(n^2\right), \]
since in this case every distance can occur at most $\lfloor k/2 \rfloor -1$ times.
A result of Erd\H os and Gy\'arf\'as \cite{EG97} implies
\[ \phi\left(n,k,\binom{k}{2}-\lfloor k/2 \rfloor +1\right) = \Omega\left(n^{4/3}\right). \]
That is, the boundary between a trivial problem and a non-trivial one passes between $\ell \ge \binom{k}{2}-\lfloor k/2 \rfloor +2$ and $\ell \le \binom{k}{2}-\lfloor k/2 \rfloor +1$.

Recently, Fox, Pach, and Suk \cite{FPS17} showed that for any $\eps>0$
\[ \phi\left(n,k,\binom{k}{2}-k+6\right) = \Omega\left(n^{8/7-\eps}\right).\]

We will prove the following by using higher distance energies.
\begin{theorem} \label{th:LocalProp}
For any integers $c,d \ge 2$ we have\footnote{Here and in the following theorems and lemmas, the hidden constant of the asymptotic notation depends on the constants defined in the theorem. For example, the $\Omega(\cdot)$-notation in the bound of Theorem \ref{th:LocalProp} depends on $c$ and $d$.}
\[ \phi\left(n,c(d+1),\binom{c(d+1)}{2} - dc + (d+1) \right) = \Omega\left(n^{1+\frac{1}{d}}\right). \]
\end{theorem}

For example, by applying Theorem \ref{th:LocalProp} with $d=2$ we get the bound
\[ \phi\left(n,k,\binom{k}{2} - 2k/3 + 7 \right) = \Omega\left(n^{3/2}\right). \]
(We have +7 rather than +3 due to cases where $k$ is not divisible by 3.)
While there are many problems in which the conjectured number of distinct distances is about $n^2$, we are very far from deriving this bound for any of those.
As far as we know, the above is the first case where a bound stronger than $\Omega(n^{4/3})$ is obtained for a non-trivial distinct distances problem.

Recall that Erd\H os and Gy\'arf\'as derived a bound of $\Omega(n^{4/3})$ distinct distances when $\ell = \binom{k}{2}-\lfloor k/2 \rfloor +1$.
Theorem \ref{th:LocalProp} implies this bound already when $\ell \ge \binom{k}{2} - 3k/4 + 10$.
Finally, Theorem \ref{th:LocalProp} leads to a better bound than the one of Fox, Pach, and Suk \cite{FPS17} whenever $\ell \ge \binom{k}{2} - 7k/8 + 22$.

The proof of Theorem \ref{th:LocalProp} is presented in Section \ref{sec:DDlocal}.

\parag{Expanding polynomials with structure.}
Given polynomials $f\in \RR[x]$ and $g\in \RR[x,y]$, and sets $A,B\subset \RR$, we write
\[ f(A) = \{f(a) :\ a\in A \} \ \text{ and } \ g(A,B) = \{g(a,b) :\ a\in A,\ b\in B\}. \]
That is, $g(A,B)$ is the set of distinct values that can be obtained by applying $g$ on the lattice $A\times B$.

Elekes and R\'onyai \cite{ER00} proved that $|f(A,B)|$ must be large, unless the polynomial has one of the special forms $f = h(g_1(x)+g_2(y))$ and $f = h(g_1(x)\cdot g_2(y))$, for some $h,g_1,g_2\in \RR[x]$.
The current best bound for this problem is the following one by Raz, Sharir, and Solymosi \cite{RSS16}.

\begin{theorem} \label{th:ElekesRonyai}
Let $A, B \subset \RR$ be finite sets, and let $f\in \RR[x,y]$ be of a constant-degree.
Then, unless $f = h(g_1(x)+g_2(y))$ or $f = h(g_1(x)\cdot g_2(y))$ for some $h,g_1,g_2\in \RR[x]$, we have
\[ f(A,B) = \Omega\left(\min\left\{|A|^{2/3}|B|^{2/3},|A|^2,|B|^2\right\}\right).\]
\end{theorem}

Theorem \ref{th:ElekesRonyai} generalizes many problems from Discrete Geometry and Additive Combinatorics, and thus has many applications (for example, see \cite{RSS16}).

Given a finite set $A\subset \RR$, the \emph{difference set} of $A$ is defined as $A-A = \left\{a-a' :\ a,a'\in A \right\}$.
When $A-A$ is small, we say that the set $A$ has an ``additive structure'' (for details about the meaning of this structure, see for example \cite[Chapter 2]{TV06}).
We now derive a stronger Elekes-R\'onyai bound when the sets have such an additive structure.
We say that a polynomial $f \in \RR[x,y]$ is \emph{additively degenerate} if $f= g \circ L$ for some $g\in \RR[z]$ and a linear $L \in \RR[x,y]$.

\begin{theorem} \label{th:StructureElekRon} Let $A,B \subset \RR$ be finite sets and let $f \in \RR[x,y]$ be a polynomial of degree at most $d$ that is not additively degenerate.
Then for any $\eps>0$ we have
\[ |f(A,B)| = \Omega\left(\min\left\{\frac{|A|^{16/9-\eps} |B|^{16/9-\eps} }{|A-A||B-B|},|A|^2,|B|^2\right\}\right). \]
\end{theorem}

In the extreme case of $|A-A|= \Theta(|A|)$ and $|B-B|= \Theta(|B|)$, Theorem \ref{th:StructureElekRon} leads to the improved bound $|f(A,B)| = \Omega\left(|A|^{7/9-\eps}|B|^{7/9-\eps} \right)$.
More generally, Theorem \ref{th:StructureElekRon} is stronger than Theorem \ref{th:ElekesRonyai} when $|A-A||B-B|= O(|A|^{10/9-\eps}|B|^{10/9-\eps})$ (both theorems give the same bound when $A$ and $B$ are of significantly different sizes).
Moreover, Theorem \ref{th:StructureElekRon} holds for a larger family of polynomials in $\RR[x,y]$.

The above result holds only for sets with an additive structure.
We can often generalize that result to a much broader definition of structure.
We say that a polynomial $p\in \RR[x,y]$ is \emph{decomposable} if there exists a univariate polynomial $p_1$ of degree at least two and $p_2\in \RR[x,y]$ such that $p(x, y) = p_1(p_2(x, y))$.
A polynomial that is not decomposable is said to be \emph{indecomposable}.
Given a polynomial $\tau\in \RR[x]$, we say that a function $\phi: \tau(\RR) \to \RR$ is a \emph{one-sided inverse} of $\tau$ if it satisfies the following: For every $a\in \tau(\RR)$ there exists $b\in \RR$ such that $\tau(b) = a$ and $\phi(a) = b$.
That is, we have that $f(f^{-1}(x)) = x$ for every $x\in \tau(\RR)$ but not necessarily $f^{-1}(f(y)) = y$ for every $y\in \RR$.
Note that $\tau$ is not required to be injective and the one-sided inverse is not necessarily unique.

\begin{theorem} \label{th:GeneralStructure}
Let $A,B \subset \RR$ be finite sets and let $f \in \RR[x,y]$ be of degree at most $d$.
Let $\tau_A,\tau_B \in \RR[x]$ be of degree at most $d$, let $\deg \tau_B \ge 2$, and let $\tau_A^{-1}$ and $\tau_B^{-1}$ be respective one-sided inverses.
Assume that $f(\tau_A(x),\tau_B(y))$ is indecomposable, that $A \subset \tau_A(\RR)$, and that $B \subset \tau_B(\RR)$.
Then for any $\eps>0$ we have
\[ |f(A,B)| = \Omega\left(\min\left\{\frac{|A|^{16/9-\eps} |B|^{16/9-\eps} }{|\tau^{-1}_A(A)-\tau^{-1}_A(A)||\tau^{-1}_B(B)-\tau^{-1}_B(B)|},|A|^2,|B|^2\right\}\right). \]
\end{theorem}

To get some intuition for Theorem \ref{th:GeneralStructure}, consider the case where
\[ A=B = \left\{\frac{1^5-1^2}{2}, \frac{2^5-2^2}{2}, \frac{3^5-3^2}{2}, \ldots, \frac{n^5-n^2}{2} \right\}. \]
By setting $\tau_A(x)=\tau_B(x) = (x^5-x^2)/2$ we get $f(A,B) = \Omega(n^{14/9-\eps})$ for any $f \in \RR[x,y]$ for which $f((x^5-x^2)/2,(y^5-y^2)/2)$ is indecomposable.
More generally, this argument holds for any $A$ and $B$ that contain large subsets of $\tau_A(\{1,2,3,\ldots\})$ and $\tau_B(\{1,2,3,\ldots\})$, as long as $f(\tau_A(x),\tau_B(y))$ is indecomposable and $\tau_B$ is nonlinear.
We thus get a good expansion for sets with a variety of types of polynomial structure.
Note that in some cases we also get an expansion for the special forms $f = h(g_1(x)+g_2(y))$ and $f =h(g_1(x)\cdot g_2(y))$.

Asking for $f(\tau_A(x),\tau_B(y))$ to be indecomposable may seem restrictive, but it is not difficult to find interesting applications with this restriction.
For example, the problem of distinct distances between two lines (see Section \ref{sec:ElekRon}) can be reduced to an expansion problem where $f(\tau_A(x),\tau_B(y))$ is indecomposable for \emph{any} $\tau_A$ and $\tau_B$.

One expects $f(A,B)$ to be larger when $A$ and $B$ do not have much structure.
For example, we expect to obtain non-trivial upper bounds for $f(A,B)$ by taking sets $A$ and $B$ that have some type of structure.
Theorems \ref{th:StructureElekRon} and \ref{th:GeneralStructure} are somewhat surprising in the sense that they shows that $f(A,B)$ is large when $A$ and $B$ do have structure.
One possible approach for improving Theorem \ref{th:ElekesRonyai} is to prove that $f(A,B)$ is large when $A$ and $B$ have no structure.
Surprisingly, this seems to be the harder case.

\ignore{ 
We can also generalize Theorem \ref{th:StructureElekRon} by replacing $|A-A|$ with $|A+A|$, $|AA|$, and $|A/A|$.
We say that a polynomial $f \in \RR[x,y]$ is \emph{multiplicatively degenerate} if $f= g \circ T$ for some $T \in \RR[x,y]$ that consists of a single monomial and $g\in \RR[z]$.

\begin{theorem} \label{th:MultStructure}
In Theorem \ref{th:StructureElekRon} we can replace the minus signs in $A-A$ and $B-B$ with addition, multiplication, or division.
In the cases of multiplication and division, $f$ is required not to be multiplicatively degenerate rather than not additively degenerate.
\end{theorem}
} 

Our techniques allow the derivation of many additional related results, such as for the case where only one of the two sets has structure and cases where the sets have a multiplicative structure.
We decided not to include such additional results, so as not to make this work too repetitive.

Section \ref{sec:ElekRon} contains a proof of Theorems \ref{th:StructureElekRon} and \ref{th:GeneralStructure}, and discusses a couple of applications of these theorems.

\parag{Bipartite distinct distances.}
In a bipartite distinct distances problem we have two point sets $\pts_1,\pts_2$ and are interested only in distances between a point from $\pts_1$ and another from $\pts_2$.
One basic bipartite problem is when $\pts_1,\pts_2\subset \RR^2$, all of the points of $\pts_1$ are on a given line, and the points of $\pts_2$ are unrestricted.
Elekes \cite{Elekes95} derived the following non-intuitive result for this problem.

\begin{theorem} \label{th:CircLattice}
If the positive integers $n$ and $m$ satisfy $n\ge 4m^3$, then there exist a set $\pts_1$ of $m$ points on a line $\ell$ in $\RR^2$ and a set $\pts_2$ of $n$ unrestricted points in $\RR^2$ such that $D(\pts_1,\pts_2) = O(m^{1/2}n^{1/2})$.
\end{theorem}

One non-intuitive consequence of Theorem \ref{th:CircLattice} is that the number of distances between $n$ points and a constant number of points can be as small as $\Theta(n^{1/2})$.
Elekes asked how close the bound of Theorem \ref{th:CircLattice} is to being tight.
As far as we know, previously no  non-trivial results were known for this question.
Brunner and Sharir \cite{BS16} derived a lower bound on the number of distinct distances between a set of points on a line and another point set.
However, their second point set has additional restrictions that forbid Elekes' construction.

By relying on a higher distance energies, we derive the following bound.

\begin{theorem} \label{th:BipartiteLineUnrest}
Let $\pts_1$ be a set of $m$ points on a line $\ell$ in $\RR^2$ and let $\pts_2$ be set of $n$ points $\RR^2$.
Then
\begin{equation*}
D(\pts_1,\pts_2) =
\begin{cases}
\Omega(m^{1/2}n^{1/2} \log^{-1/2}n), \hspace{8mm} \text{\emph{when }}m=\Omega(n^{1/2}/\log^{1/3} n),\\
\Omega\left(n^{3/8}m^{3/4}\right), \hspace{22mm} \text{\emph{when }}m=O(n^{1/2}/\log^{1/3} n) \text{ \emph{and} }  m=\Omega(n^{3/10}),\\
\Omega\left(n^{1/2}m^{1/3}\right), \hspace{22mm} \text{\emph{when }}m=O(n^{3/10}).
\end{cases}
\end{equation*}
\end{theorem}

Note that when $m=\Omega(n^{1/2}/\log^{1/3} n)$, Theorem \ref{th:BipartiteLineUnrest} matches Elekes' bound $O(m^{1/2}n^{1/2})$ up to the $\sqrt{\log n}$.
However, Elekes' construction only holds in the separate range $m=O(n^{1/3})$.
It is not clear whether similar constructions exist for larger values of $m$, and it is also possible that when $m> (n/4)^{1/3}$ the number of distinct distances jumps to $\Omega(n/\sqrt{\log n})$.

While it is conjectured that every set of $n$ points in $\RR^2$ spans $\Omega(n/\sqrt{\log n})$ distinct distances, the Guth--Katz analysis \cite{GK15} leads to the slightly weaker bound $\Omega(n/\log n)$.
In Theorem \ref{th:BipartiteLineUnrest}, when $m=\Omega(n^{1/2}/\log^{1/3} n)$ the distinct distances bound does contain $\sqrt{\log n}$.
In fact, the bound $D(\pts_1,\pts_2) = \Omega(m^{1/2}n^{1/2} \log^{-1/2}n)$ is what one might expect to obtain for the general bipartite variant of the distinct distances problem.
This leads to asking whether the third distance energy could lead to such a distinct distances bound in more general cases.

A proof of Theorem \ref{th:BipartiteLineUnrest} can be found in Section \ref{sec:Bipartite}.

\parag{Subsets with few repeating distances.}
Erd\H os \cite{Erd57,EG70} made the following conjecture.

\begin{conjecture}\label{co:subset}
Let $\pts$ be a set of $n$ points in $\RR^2$.
Then there exists a subset $\pts' \subset \pts$ such that $|\pts'|=\Omega(n^{1/2})$ and no distance is spanned by two pairs of points of $\pts'$.
\end{conjecture}

In other words, the conjecture suggests that every planar point set contains a large subset with no repeating distances.
Charalambides \cite{Chara12} showed that the Guth--Katz result \cite{GK15} implies the existence of a subset of size $\Omega(n^{1/3}\log^{-1/3}n)$ with no repeating distances.
This is the current best bound.
By using higher distance energies, we can show that there exists a larger subset with no distance repeating more than twice.

\begin{theorem} \label{th:Subset}
Let $\pts \subset \RR^2$ be a set of $n$ points.
Then there exists a subset $\pts' \subset \pts$ of size $\Omega(n^{22/63} \log^{-13/63}n)$ such that no distance is spanned by more than four pairs of points of $\pts'$.
Similarly, there exists a subset $\pts' \subset \pts$ of size $\Omega(n^{12/35} \log^{-9/35} n)$ such that no distance is spanned by more than two pairs of points of $\pts'$.
\end{theorem}

Once again, our techniques can lead to a variety of related results but we decided not to include too many similar results.
The proof of Theorem \ref{th:Subset} and one of its variants can be found in Section \ref{sec:Subset}.
\vspace{2mm}

\parag{Acknowledgements.} The second author would like to thank Micha Sharir for several useful discussions.

\section{Higher distance energies} \label{sec:HigerEner}

For $a,b\in \RR^2$, we denote by $|ab|$ the Euclidean distance between these two points.
Let $\pts$ be a set of $n$ points in $\RR^2$ and let $d\ge 1$ be an integer.
We define the $d$'\emph{th distance energy} of $\pts$ as
\begin{equation} \label{eq:IntHigherEnergyDef}
E_d(\pts) = \left|\left\{(a_1,b_1,\ldots,a_d,b_d)\in \pts^{2d} :\ |a_1b_1| = \cdots = |a_db_d| >0\right\}\right|.
\end{equation}
The $2d$-tuples are ordered, so  $(a_1,b_1,a_2,b_2,\ldots,a_d,b_d)$ and $(b_1,a_1,a_2,b_2,\ldots,a_d,b_d)$ are considered as two distinct tuples in the above set.

Let $\Delta = \Delta(\pts)$ be the set of distances that are spanned by pairs of points of $\pts$.
For a distance $\delta \in \Delta$, we denote by $m_\delta$ the number of ordered pairs $(a,b)\in \pts^2$ such that $|ab|=\delta$.
Note that the number of $2d$-tuples in \eqref{eq:IntHigherEnergyDef} that correspond to a specific $\delta\in \Delta$ is exactly $m_\delta^d$.
This implies that
\begin{equation} \label{eq:GenHigherEnergyDef}
E_d(\pts) = \sum_{\delta\in \Delta}m_\delta^d.
\end{equation}

Since every ordered pair of distinct points $(a,b)\in \pts^2$ contributes to exactly one $m_\delta$, we have that $\sum_{\delta\in \Delta} m_\delta = 2\binom{n}{2} = n^2-n$.
Let $D=D(\pts) = |\Delta|$ be the number of distinct distances spanned by point of $\pts$.
By H\"older's inequality, for any $d\ge 2$ we have
\begin{equation} \label{eq:dthEnergyLower}
E_{d}(\pts) = \sum_{\delta \in \Delta}{m_{\delta}^{d}}\geq \frac{\left(\sum_{\delta \in \Delta}{m_{\delta}}\right)^{d}}{D^{d-1}}=\frac{\left(n^2-n\right)^d}{D^{d-1}} = \Omega\left(\frac{n^{2d}}{D^{d-1}}\right).
\end{equation}

Using somewhat different notation, Guth and Katz \cite{GK15} derived the tight upper bound
\begin{equation} \label{eq:GuthKatzE2}
E_{2}(\pts) = O\left(n^3 \log n\right).
\end{equation}

Also using different notation, Spencer, Szemer\'edi, and Trotter \cite{SST84} proved that every $\delta\in \Delta$ satisfies
\begin{equation} \label{eq:UnitDist}
m_\delta = O\left(n^{4/3}\right).
\end{equation}
For any integer $d\ge 2$, combining \eqref{eq:GuthKatzE2} and \eqref{eq:UnitDist} implies
\[ E_d(\pts) = \sum_{\delta\in \Delta}m_\delta^d = O\left(\left(n^{4/3}\right)^{d-2} \sum_{\delta\in \Delta}m_\delta^2 \right) = O\left(n^{(4d+1)/3}\log n\right).
\]

It seems reasonable to make the following conjecture

\begin{conjecture} \label{co:HigherDistEnergy}
For every set $\pts$ of $n$ points in $\RR^2$, integer $d\ge 2$, and $\eps>0$,
\[ E_d(\pts) = O(n^{d+1+\eps}). \]
\end{conjecture}

The unit distances conjecture suggests that $m_\delta = O(n^{1+\eps})$ for every $\delta\in \Delta$, which would immediately imply conjecture \ref{co:HigherDistEnergy}.
In Section \ref{sec:Subset} we use geometric incidences to prove the stronger bound $E_3(\pts) = O\left(n^{30/7}\log^{9/7} n\right)$.
Similar techniques lead to improved bounds for $E_d(\pts)$ when $d\ge 4$.
Further improving these bounds would immediately improve the results of Section \ref{sec:Subset}, and possibly other parts of this paper.

\parag{Distance energy variants.}
For finite sets $\pts_1,\pts_2\subset \RR^2$ and an integer $d\ge 2$, we define the $d$'th distance energy of $\pts_1$ and $\pts_2$ as
\[ E_{d}(\pts_1,\pts_2) = \left\{(a_1,\ldots,a_d,b_1,\ldots,b_d)\in \pts_1^d \times \pts_2^d :\ |a_1b_1|=\cdots = |a_db_d|>0\right\}. \]

Note that in this case we are only interested in distances between points from different sets, and ignore the distances between points in the same set.
This corresponds to a bipartite distinct distances problem, as defined in the introduction.
For such problems we define $\Delta = \Delta(\pts_1,\pts_2)$ to be the set of distances spanned by pairs of $\pts_1\times \pts_2$.
The number of distinct distances is defined accordingly as $D= D(\pts_1,\pts_2) = |\Delta|$.
By imitating the argument that led to \eqref{eq:dthEnergyLower}, we obtain the bound
\begin{equation} \label{eq:dthEnergyLowerBipartite}
E_{d}(\pts_1,\pts_2) = \Omega\left(\frac{|\pts_1|^d|\pts_2|^d}{D^{d-1}}\right).
\end{equation}

Finally, for a point set $\pts\subset \RR^2$ and a positive integer $d\ge2$, we define
\[ E_d^*(\pts) = \left\{(a_1,b_1,\ldots,a_d,b_d)\in \pts^{2d} :\ |a_1b_1| = \cdots = |a_db_d| \text{ and the } 2d \text{ points are distinct} \right\}. \]
In other words, $E_d^*(\pts)$ is a variant of $E_d(\pts)$ that considers only tuples with $2d$ distinct elements.
By definition $E_d^*(\pts) < E_d(\pts)$.

\section{Geometric preliminaries}

\parag{Tools from Discrete Geometry.}
Given a set $\pts$ of points and a set $\curves$ of curves in  $\RR^2$ (such as lines, circles, or sine waves), an \emph{incidence} is a pair $(p,\gamma)\in \pts \times \curves$ such that the point $p$ is contained in the curve $\gamma$.
The number of incidences in $\pts\times\curves$ is denoted as $I(\pts,\curves)$.
The \emph{incidence graph} of $\pts\times\curves$ is a bipartite graph $G=(\pts \cup \curves,E)$, where an edge $(v_j,v_k)\in E$ implies that the point that corresponds to $v_j$ is incident to the curve that corresponds to $v_k$; that is, there is a bijection between the edges of $E$ and the incidences in $\pts\times\curves$.
The following incidence bound is by Pach and Sharir \cite{PS92}.

\begin{theorem} \label{th:PachSharir}
Let $\pts$ be a set of $m$ points and let $\curves$ be a set of $n$ distinct irreducible algebraic curves of degree at most $k$, both in $\RR^2$.
If the incidence graph of $\pts\times\curves$ contains no copy of $K_{2,t}$, then
\[ I(\pts,\curves) = O\left(m^{2/3}n^{2/3}+m+n\right). \]
\end{theorem}

The following incidence bound is a combination of results from several papers (for example, see Sharir and Zahl \cite{SZ16}).

\begin{theorem} \label{th:IncSharirZahl}
Let $\pts$ be a set of $m$ points and let $\curves$ be a set of $n$ circles, both in $\RR^2$.
Then
\[ I(\pts,\curves) = O\left(m^{6/11}n^{9/11}\log^{2/11}n+m^{2/3}n^{2/3}+m+n\right).\]
\end{theorem}

The following incidence result is by Sharir and Zahl \cite{SZ16}.

\begin{theorem} \label{th:SharirZahl}
Let $\pts$ be a set of $m$ points and let $\curves$ be a set of $n$ irreducible algebraic curves of degree at most $k$ in $\RR^2$. Assume that we can parameterize these curves using $s$ parameters.
Then for every $\eps>0$ we have
\[ I(\pts,\curves) = O\left(m^{\frac{2s}{5s-4}+\eps}n^{\frac{5s-6}{5s-4}}+m^{2/3}n^{2/3}+m+n\right).\]
\end{theorem}

Note that Theorem \ref{th:SharirZahl} almost generalizes Theorem \ref{th:IncSharirZahl}, except that $m^\eps$ is asymptotically larger than $\log^{2/11}n$ (for the range of $m$ and $n$ in which the term $m^{6/11}n^{9/11}$ dominates the bound).

We will rely on the following distinct distances result of Pach and de Zeeuw \cite{PdZ17}.

\begin{theorem} \label{th:PachdeZeeuw}
Let $\gamma_1$ and $\gamma_2$ be two distinct irreducible algebraic curves of degree at most $d$ in $\RR^2$, which are not
parallel lines, orthogonal lines, or concentric circles.
Let $\pts_1$ be a set of $m$ points that are incident to $\gamma_1$ and let $\pts_2$ be a set of $n$ points incident to $\gamma_2$.
Then
\[ D(\pts_1,\pts_2) = \Omega\left( \min\left\{m^{2/3}n^{2/3},m^{2}, n^{2}\right\}\right). \]
\end{theorem}

For a point set $\pts \subset \RR^2$, let $t(\pts)$ denote the number of isosceles triangles that are spanned by $\pts$ (that is, isosceles triangles that have their three vertices in $\pts$). The following result is by Pach and Tardos \cite{PT02}.

\begin{theorem} \label{th:Isosceles}
Let $\pts$ be a set of $n$ points in $\RR^2$.
Then $t(\pts) = O(n^{2.137})$
\end{theorem}

\parag{Tools from Algebraic Geometry.}
For a polynomial $f\in \RR[x_1,\ldots,x_d]$ we denote by $\vb(f)$ the variety defined by $f$ (that is, the set of points in $\RR^d$ on which $f$ vanishes).

\begin{theorem}[Milnor--Thom \cite{Mil64}] \label{th:Milnor}
Let $f_1,\ldots,f_m\in \RR[x_1,\ldots,x_d]$ be of degree at most $k$.
Then the number of connected
components of $\vb(f_1,\ldots,f_m)$ is at most
\[k(2k-1)^{d-1}. \]
\end{theorem}

A polynomial $f\in \RR[x,y]$ is said to be \emph{reducible} if there exist polynomials $f_1,f_2\in \RR[x,y]$ of positive degrees such that $f(x,y) = f_1(x,y) \cdot f_2(x,y)$.
A polynomial that is not reducible is said to be \emph{irreducible}.
The following result is a combination of \cite{Ayad02} and \cite{Stein89} (for more details, see for example \cite{RSS16}).

\begin{theorem} \label{th:Stein}
If $f\in \RR[x,y]$ is indecomposable, then the polynomial $f(x,y) - \lambda$ is reducible for at most $\deg f$ values of $\lambda\in \RR$.
\end{theorem}

Finally, we will rely on the following Schwartz--Zippel variant in $\RR^2$.

\begin{lemma} \label{le:SchZipp}
Let $f\in \RR[x,y]$ be a polynomial of degree $d>0$, and let $A,B \subset \RR$ be finite sets.
Then
\[ |\vb(f)\cap (A\times B)| = O(|A|+|B|). \]
\end{lemma}

\section{Distinct Distances with Local Properties}\label{sec:DDlocal}

In this section we prove Theorem \ref{th:LocalProp}.
We begin by recalling the statement of this theorem.
\vspace{1mm}

\noindent {\bf Theorem \ref{th:LocalProp}.}
\emph{For any integers $c,d \ge 2$ we have}
\[ \phi\left(n,c(d+1),\binom{c(d+1)}{2} - dc + (d+1) \right) = \Omega\left(n^{1+\frac{1}{d}}\right). \]

To prove the theorem, we will rely on the following simple counting lemma (for example, see \cite[Lemma 2.3]{Jukna11}).

\begin{lemma} \label{le:GeneralSetIntersection}
Let $A$ be a set of $n$ elements and let $d\ge 2$ be an integer.
Let $A_1,\ldots,A_k$ be subsets of $A$, each of size at least $m$.
If $k \ge 2d n^d/m^d$ then there exist $1\le j_1 < \ldots < j_d \le d$ such that $|A_{j_1}\cap \ldots \cap A_{j_d}| \ge \frac{m^d}{2n^{d-1}}$.
\end{lemma}

\begin{proof}[Proof of Theorem \ref{th:LocalProp}.]
Let $\pts$ be a set of $n$ points such that every $c(d+1)$ points of $\pts$ span at least  $\binom{c(d+1)}{2} - dc + (d+1)$ distinct distances.
We say that a point $p\in \pts$ \emph{spans} a distance $\delta$ if there exists a point $q\in \pts$ such that $|pq|=\delta$.
We begin by studying some configurations that cannot occur in $\pts$.

\parag{Forbidden configurations.}
Assume that a point $p\in \pts$ is at a distance of $\delta$ from at least $dc-d+1$ points of $\pts$.
Let $\pts' \subset \pts$ consist of $p$, of $dc-d+1$ points of $\pts$ that are at a distance of $\delta$ from $p$, and of $d+c-2$ additional points of $\pts$.
Then $\pts'$ is a set of $c(d+1)$ points of $\pts$ that span at most $\binom{c(d+1)}{2} - dc + d$ distinct distances, contradicting the assumption on $\pts$.
This contradiction implies that for every $p\in \pts$ and distance $\delta$, at most $dc-d$ points of $\pts$ are at a distance of $\delta$ from $p$.
This in turn implies that every distance is spanned by at most $d(c-1)n/2$ pairs of $\pts^2$.

Let $\delta_1,\ldots,\delta_d$ be distinct distances and let $\pts' \subset \pts$ be a set of $c$ points, such that every $p\in \pts'$ spans each of the distances $\delta_1,\ldots,\delta_d$ with a point of $\pts\setminus \pts'$.
We construct a set $\pts''$ by going over every point $p\in \pts$ and distance $\delta_j$, and adding to $\pts''$ one point of $\pts\setminus \pts'$ that is at a distance of $\delta_j$ from $p$.
If a point $q \in \pts\setminus \pts'$ was added several times to $\pts''$, we consider $\pts''$ as containing one instance of $q$ (that is, $\pts''$ is not a multiset).
We get that $\pts' \cup \pts''$ is a set of at most $c(d+1)$ points, with at least $c$ pairs of points at distance $\delta_j$ for every $1\le j \le d$.
That is, the set $\pts' \cup \pts'' \subset \pts$ spans at most $\binom{c(d+1)}{2} - dc + d$ distinct distances.
This contradicts the assumption on $\pts$, and implies that $\pts$ cannot contain such a configuration.

Next, we slightly change the case studied in the previous paragraph: We assume that every point of $\pts'$ spans the distances $\delta_1,\ldots,\delta_d$ with points of $\pts$ (rather than of $\pts\setminus \pts'$) and that each of the distances $\delta_1,\ldots,\delta_d$ is spanned by at least $c$ pairs of points of $\pts$.
We construct the set $\pts'' \subset \pts$ as in the previous paragraph: For every $p\in \pts$ and distance $\delta_j$, if there exist points of $\pts\setminus \pts'$ at distance $\delta_j$ from $p$ then we add one such point to $\pts''$.
For $1 \le j \le d$, let $k_j$ denote the number of points of $\pts'$ that are not at distance $\delta_j$ from any point of $\pts\setminus \pts'$.
By the previous paragraph, the case where $k_1 = \cdots = k_d =0$ is forbidden.
Note that for every $1 \le j \le d$, at least $k_j/2$ pairs of points of $\pts'$ are at distance $\delta_j$.
That is, $\pts^* = \pts' \cup \pts''$ is a set of at most $(d+1)c-\sum_{j=1}^d k_j$ points of $\pts$ and at least $c-k_j/2$ pairs of points at distance $\delta_j$.
For $1 \le j \le d$, we add $\lfloor k_j/2 \rfloor$ additional pairs at distance $\delta_j$ by adding at most $k_j$ vertices of $\pts$ to $\pts^*$.
This is possible by the assumption that $\delta_j$ is spanned by at least $c$ pairs of points of $\pts$.
Since the resulting set $\pts^*$ contradicts the assumption about $\pts$, this configuration is also forbidden.

\parag{Rich distances.}
For an integer $j$, let $\Delta_j$ be the set of distances spanned by at least $j$ pairs of points of $\pts$, and set $k_j = |\Delta_j|$.
Let $\pts_\delta$ be the set of points of $\pts$ that span $\delta$.

Fix $j\ge d(2c^{d+1}n^{d-1})^{1/d}$ and let $\delta \in \Delta_j$.
Since every point spans $\delta$ with at most $dc-d$ points of $\pts$, we get that $|\pts_\delta|\ge \frac{2j}{dc-d} > \frac{j}{dc}$.
By the last forbidden configuration above, for any choice of $\delta_1, \ldots, \delta_d \in \Delta_j$ we have $|\pts_{\delta_1} \cap \cdots \cap \pts_{\delta_d}| <c$.
By the assumption on $j$ we have $c \le \frac{j^d}{2c^d d^d n^{d-1}}$.
We may thus apply the contrapositive of Lemma \ref{le:GeneralSetIntersection} on the sets $\pts_\delta$ with $\delta\in \Delta_j$ and with $m=\frac{j}{dc}$.
This implies
\begin{equation} \label{eq:RichDistUpper}
k_j < \frac{2dn^d}{(j/dc)^d} = \frac{2n^d d^{d+1}c^d}{j^d}.
\end{equation}

For $j< d(2c^{d+1}n^{d-1})^{1/d}$ (and also for larger values of $j$), we have the straightforward bound
\begin{equation} \label{eq:PoorDistUpper}
k_j < n^2/j.
\end{equation}

\parag{An energy argument.}
Recall the notation $\Delta$, $D$, and $m_\delta$ from Section \ref{sec:HigerEner}.
Since no distance is spanned by more than $d(c-1)n/2$ pairs of points of $\pts$, we have that $m_\delta \le d(c-1)n$ for every $\delta\in \Delta$.
Let $m = \lfloor \log d(2c^{d+1}n^{d-1})^{1/d} \rfloor$.
By dyadic pigeonholing together with \eqref{eq:RichDistUpper} and \eqref{eq:PoorDistUpper}, we obtain
\begin{align*}
E_d(\pts) = \sum_{\delta\in \Delta} m_\delta^d &= \sum_{j=0}^{\log (dcn)} \sum_{\delta\in \Delta \atop 2^{j} \le m_\delta <2^{j+1} } m_\delta^d < \sum_{j=0}^{\log (dcn)} k_{2^j} \left(2^{j+1}\right)^d \\
&= \sum_{j=0}^{m} k_{2^j} 2^{d(j+1)} + \sum_{j=m+1}^{\log (dcn)} k_{2^j} 2^{d(j+1)} \\[2mm]
&= O\left(n^{2+(d-1)^2/d}\right)+O\left(n^d \log n\right) = O\left(n^{2+(d-1)^2/d}\right).
\end{align*}

Recall from \eqref{eq:dthEnergyLower} that $E_{d}(\pts) = \Omega\left(\frac{n^{2d}}{D^{d-1}}\right)$.
Combining these two bounds on $E_{d}(\pts)$ yields the assertion of the theorem.
\end{proof}

One way to improve the proof of Theorem \ref{th:LocalProp} might be to derive an upper bound on $k_j$ stronger than the straightforward bound of \eqref{eq:PoorDistUpper} when $j\approx n^{(d-1)/d}$.

\section{Elekes--Ronyai bounds with additive structure} \label{sec:ElekRon}

In this section we prove Theorems \ref{th:StructureElekRon} and \ref{th:GeneralStructure}, and discuss some of their applications.
We begin by recalling the statement of Theorem \ref{th:StructureElekRon}.
\vspace{1mm}

\noindent {\bf Theorem \ref{th:StructureElekRon}.} \emph{Let $A,B \subset \RR$ be finite sets and let $f \in \RR[x,y]$ be a polynomial of degree at most $d$ that is not additively degenerate.
Then for any $\eps>0$ we have}
\[ |f(A,B)| = \Omega\left(\min\left\{\frac{|A|^{16/9-\eps} |B|^{16/9-\eps} }{|A-A||B-B|},|A|^2,|B|^2\right\}\right). \]

Theorem \ref{th:StructureElekRon} has many applications, and we now present a couple of those.
We begin with the following distinct distances result from \cite{SSS13}.

\begin{theorem} \label{th:DDtwoLines}
Let $\ell_1$ and $\ell_2$ be two lines in $\RR^2$ that are neither parallel nor orthogonal.
Let $\pts_1$ and $\pts_2$ be sets of points in $\RR^2$ of size $n$ and $m$ respectively, such that $\pts_1$ is contained in $\ell_1$, and $\pts_2$ is contained in $\ell_2$.
Then
\[ D(\pts_1,\pts_2) = \Omega\left(\min\left\{n^{2/3}m^{2/3}, n^2, m^2\right\}\right). \]
\end{theorem}

One reason for why the above distinct distances problem is considered interesting is that it has many generalizations (such as Theorem \ref{th:ElekesRonyai}).
 Improving the known bounds for the distances problem tends to lead to improvements for various generalizations.
Quoting Hilbert \cite{Reid70}: ``The art of doing mathematics is finding that special case that contains all the germs of generality.''
The following is an easy corollary of Theorems \ref{th:StructureElekRon} and \ref{th:GeneralStructure}.

\begin{corollary} $\quad$

(a) Let $\ell_1$ and $\ell_2$ be two lines in $\RR^2$ that are neither parallel nor orthogonal.
Let $\pts_1$ and $\pts_2$ be sets of points in $\RR^2$ of size $n$ and $m$ respectively, such that $\pts_1\subset \ell_1$ and $\pts_2 \subset \ell_2$.
Let $A$ be the set of $x$-coordinates of the points of $\pts_1$, and let $B$ be the set of $x$-coordinates of the points of $\pts_2$ (assuming the neither line is parallel to the $y$-axis).
Then
\[ D(\pts_1,\pts_2) = \Omega\left(\min\left\{\frac{m^{16/9-\eps} n^{16/9-\eps} }{|A-A||B-B|},m^2,n^2\right\}\right). \]

(b) Let $\tau_A,\tau_B \in \RR[x]$ be of degree at most $d$, let $\deg \tau_B \ge 2$, and let $\tau_A^{-1}$ and $\tau_B^{-1}$ be respective one-sided inverses.
Assume that $f(\tau_A(x),\tau_B(y))$ is indecomposable, that $A \subset \tau_A(\RR)$, and that $B \subset \tau_B(\RR)$.
Then for any $\eps>0$
\[ D(\pts_1,\pts_2) = \Omega\left(\min\left\{\frac{m^{16/9-\eps} n^{16/9-\eps} }{|\tau_A^{-1}(A)-\tau_A^{-1}(A)||\tau_B^{-1}(B)-\tau_B^{-1}(B)|},m^2,n^2\right\}\right). \]
\end{corollary}
\begin{proof}[Proof sketch.]
In \cite{SSS13} the distinct distances problem is reduced to showing that the polynomial $f(x,y) = x^2-2xy+(1+s^2)y^2$ expands (where $s\in \RR\setminus\{0\}$ depends on the angle between $\ell_1$ and $\ell_2$).
It can be easily verified that $f(x,y)$ is indecomposable, so we can apply Theorem \ref{th:StructureElekRon} with it.
Plugging the resulting expansion bound in the analysis of \cite{SSS13} immediately leads to the bound of part (a).

For $\tau_A$ and $\tau_B$ of respective degrees $d$ and $d'\ge 2$, assume for contradiction that $f(\tau_A(x),\tau_B(y))$ is decomposable. 
That is, there exist $g\in \RR[x,y]$ and $h\in \RR[z]$ of degree at least two such that $f(\tau_A(x),\tau_B(y)) = h(g(x,y))$.
Note that $f(\tau_A(x),\tau_B(y))$ contains terms of the form $(c \cdot x)^{2d}$ and $(1+s^2) (c'\cdot y)^{2d'}$ where $c,c' \in \RR\setminus\{0\}$.
In every other term of $f(\tau_A(x),\tau_B(y))$ the degree of $x$ is smaller than $2d$ and the degree of $y$ is smaller than $2d'$.
This implies that $2d = \deg h \cdot \deg_x g$, and that $2d'= \deg h \cdot \deg_y g$.

By the above, $f(\tau_A(x),\tau_B(y))$ should also contain a term where $x$ is of degree $2d \cdot \frac{\deg h -1}{\deg h}$ and $y$ is of degree at least one.
By the definition of $f(x,y)$, we note that $f(\tau_A(x),\tau_B(y))$ does not contain terms with an exponent of $x$ larger than $d$ that also involve $y$.
We thus get that $\deg h = 2$. 
Let $c_2$ be the coefficient of $z^2$ in $h\in \RR[z]$.
Recalling the terms $(c \cdot x)^{2d}$ and $(1+s^2) (c'\cdot y)^{2d'}$, we get that $g(x,y)$ contains the terms $c_2^{-1/2}(c \cdot x)^{d}$ and $c_2^{-1/2}(1+s^2)^{1/2}(c' \cdot y)^{d'}$.
This in turn implies that $f(\tau_A(x),\tau_B(y))$ should also contain the term $2(c \cdot x)^{d} (1+s^2)^{1/2}(c' \cdot y)^{d'}$.
From the definition of $f(x,y)$ we note that $f(\tau_A(x),\tau_B(y))$ contains the slightly different term $2(c \cdot x)^{d} (c' \cdot y)^{d'}$.
Since $s\neq 0$ we get a contradiction, so $f(\tau_A(x),\tau_B(y))$ must be indecomposable.
Applying Theorem \ref{th:GeneralStructure} with $f(\tau_A(x),\tau_B(y))$ immediately leads to part (b).
\end{proof}

For example, when $|A-A|=\Theta(n)$ and $|B-B|=\Theta(m)$ we have $D(\pts_1,\pts_2) = \Omega\left(\min\left\{m^{7/9-\eps} n^{7/9-\eps},m^2,n^2\right\}\right)$.
One approach to improving Theorem \ref{th:DDtwoLines} is to improve the case where $A$ and $B$ have no additive or polynomial structure.

We next consider a result of Shen \cite{Shen12}.

\begin{theorem} \label{th:Shen}
Let $f\in \RR[x,y]$ be a polynomial of a constant-degree that is not additively degenerate, and let $A\subset \RR$ be finite.
Then
\[ |A + A| + |f(A, A)| = \Omega\left(|A|^{5/4}\right). \]
\end{theorem}

Theorem \ref{th:StructureElekRon} implies a better lower bound for $|f(A,A)|$ when $|A+A|=O(|A|^{83/72-\eps})$.\footnote{In Theorem \ref{th:StructureElekRon} it is easy to replace $A-A$ with $A+A$: In the proof we take $\alpha \in A+A$, replace $x+\alpha$ with $\alpha-x$, and handle $\beta$ in a symmetric manner. The rest of the proof remains the same. }
Thus, to improve the bound of Theorem \ref{th:Shen} it remains to handle the case where $|A+A|=\Omega(|A|^{83/72})$ and $|A+A|=O(|A|^{5/4})$.

Theorems \ref{th:StructureElekRon} and \ref{th:GeneralStructure} have many additional applications.
See for example the applications presented in \cite{RSS16}.

\begin{proof}[Proof of Theorem \ref{th:StructureElekRon}.]
If $f$ is decomposable, then there exist a univariate $f_1$ of degree at least two and $f_2\in \RR[x,y]$ such that $f(x,y)=f_1(f_2(x,y))$.
Let $(f_1,f_2)$ be a pair of such polynomials that minimizes the degree of $f_2$.
In particular, this implies that $f_2$ is indecomposable.
Since $f$ is of degree at most $d$, so are $f_1$ and $f_2$.
Since $f_1$ is univariate, for every $a\in \RR$ there exist at most $d$ numbers $b\in \RR$ such that $f_1(b)=a$.
Thus, if $|f_2(A,B)|\ge x$ for some $x$, then $|f(A,B)|\ge x/d$.
It then remains to derive the bound of the theorem to the indecomposable $f_2$.
Since $f$ is assumed not to be additively degenerate, so is $f_2$.
With an abuse of notation, we will refer to $f_2$ as $f$.
We may thus assume that $f$ is indecomposable.

Let $\Delta = \{f(a,b) :\ a\in A, b\in B\}$, and let $D = |\Delta|$.
For $\delta \in \Delta$, let
\[ m_\delta = |\{(a,b) \in A \times B :\ f(a,b)=\delta\}|. \]
In other words, $m_\delta$ is the number of representations of $\delta$ as $f(a,b)$.
For an integer $j \ge 1$, let $\Delta_j$ be the set of $\delta\in \Delta$ that satisfy $m_\delta \ge j$, and let $k_j=|\Delta_{j}|$.

Imitating the concept of bipartite distance energy, we define the energy
\[ E_{f}(A,B) = \left|\left\{(a_1,a_2,b_1,b_2) \in A^2 \times B^2 :\ f(a_1,b_1) = f(a_2,b_2)\right\}\right|.\]
Note that $m_\delta = |\vb(f-\delta)\cap (A\times B)|$.
By Lemma \ref{le:SchZipp}, every $\delta\in \Delta$ satisfies $m_\delta = O(|A|+|B|)$.
That is, there exists a constant $\mu$ such that $k_{\mu(|A|+|B|)} =0$.
We thus have
\begin{align} \label{eq:EfUpper}
E_{f}(A,B) = \sum_{\delta\in \Delta}m_{\delta}^2 = \sum_{j=0}^{\log \mu(|A|+|B|) } \sum_{\delta \in \Delta \atop 2^{j} \leq m_\delta < 2^{j+1}}{m_{\delta}^2} < \sum_{j=0}^{\log \mu(|A|+|B|)} 2^{2j+2} k_{2^j}.
\end{align}

Since every $(a,b)\in A\times B$ contributes to one $m_\delta$, we have that $\sum_{\delta\in \Delta} m_\delta = |A||B|$.
By the Cauchy-Schwarz inequality, we obtain
\begin{align} \label{eq:EfLower}
E_{f}(A,B) = \sum_{\delta\in \Delta}m_{\delta}^2 \ge \frac{(\sum_{\delta\in \Delta} m_\delta)^2}{D} = \frac{|A|^2|B|^2}{D}.
\end{align}

\parag{A set of curves.}
For a fixed $j\ge 2d$, consider the set of curves
\[ \curves_{j} =  \left\{\vb(f(x+\alpha,y+\beta) - \delta)\  :\ \alpha \in A-A,\ \beta \in B-B,\ \delta \in D_{j} \right\}. \]

Assume for contradiction that  there exist $(\alpha,\beta,\delta), (\alpha',\beta',\delta') \in (A-A) \times (B-B) \times D_{j}$ and $r\in \RR$ such that
\[ f(x+\alpha,y+\beta) -\delta = r\cdot \left(f(x+\alpha',y+\beta') - \delta'\right).\]
For the maximum degree terms on both sides to have the same coefficients, we must have $r=1$.
We replace $x$ with $x-\alpha'$ and $y$ with $y-\beta'$ (that is, we translate the plane by $-\alpha'$ in the $x$-direction and by $-\beta'$ in the $y$-direction). Setting $\alpha_{0}=\alpha-\alpha'$ and $\beta_{0}=\beta-\beta'$ leads to
\begin{equation} \label{eq:IdenticalCurves}
f(x+\alpha_{0},y+\beta_{0}) = f(x,y) - \delta' + \delta.
\end{equation}

One obvious solution to \eqref{eq:IdenticalCurves} is $\alpha_{0}=\beta_{0}=0$ and $\delta=\delta'$.
We now assume that this is not the case.
Without loss of generality we assume that $\beta_0 \neq 0$ and let $k$ be the degree of $x$ in $f$ (otherwise, $\alpha_0 \neq 0$ and we take $k$ to be the degree of $y$).
We claim that in this case $f(x,y)=h(x+cy)$ where $c=-\alpha_{0}/\beta_{0}$ and $h\in \RR[z]$, and prove this by induction on $k$.
For the induction basis, consider the case of $k=0$.
Since $f$ does not depend on $x$ and is of degree at least two, we get from \eqref{eq:IdenticalCurves} that $f$ is a constant function.
The claim follows by taking $h$ to be the same constant function.

For the induction step, consider $k\ge 1$.
For $0\le \ell \le k$, let $f_\ell \in \RR[y]$ be the coefficient of $x^\ell$ in $f$.
That is,
\[ f(x,y)= \sum_{\ell=0}^k x^\ell f_\ell(y). \]

From \eqref{eq:IdenticalCurves} and the assumption $\beta_0\neq 0$, we have that $f_{k}(y)$ is constant.
We set $g(x,y)= f(x,y)-(x+cy)^{k} \cdot f_{k}(y)$, and note that $g(x,y)$ is of degree $k-1$ in $x$.
By \eqref{eq:IdenticalCurves} and the definition of $c$, we have that
\begin{align*}
g(x+\alpha_{0},y+\beta_{0}) &= f(x+\alpha_{0},y+\beta_{0}) - (x+\alpha_{0}+c(y+\beta_{0}))^{k} \cdot f_{k}(y) \\
&= f(x,y) - \delta' + \delta - (x+cy)^k \cdot f_{k}(y) = g(x,y) - \delta' + \delta.
\end{align*}
We may thus apply the induction hypothesis on $g$, and obtain that $g = h'(x+cy)$ for some $h'\in \RR[z]$.
Since $g(x,y)= f(x,y)-(x+cy)^{k} \cdot f_{k}(y)$ and $f_k(y)$ is constant, we conclude that $f=h(x+cy)$ for some $h\in \RR[z]$.
This implies that $f$ is additively degenerate also before the aforementioned translation of $\RR^2$.

To recap, we proved that either $\alpha_{0}=\beta_{0}=0$ and $\delta=\delta'$, or $f$ is additively degenerate.
Since the latter contradicts the assumption, we are in the former case.
This in turn implies that $\Gamma_{j}$ is a set of $|A-A| |B-B|k_j$ curves that are defined by distinct polynomials.
Moreover, no polynomial is a constant multiple of another.
While the polynomials are distinct, the curves of $\Gamma_{j}$ may still have common components.

\parag{Studying rich $\delta$'s.}
Since $f$ is indecomposable, Theorem \ref{th:Stein} implies that there are at most $d$ values of $\delta \in \Delta$ such that $f(x,y) -\delta$ is reducible.
For any $\alpha,\beta\in \RR$, since $\vb(f(x+\alpha,y+\beta)-\delta)$ is a translation of $\vb(f(x,y)-\delta)$, either both curves are reducible or both are irreducible.
Thus, at most $d|A-A||B-B|$ curves of $\curves_j$ are reducible.

If $k_j \ge 2d$, then after removing the reducible curves from $\curves_j$ we still have $|\curves_j| =\Theta(|A-A||B-B|k_j)$.
Assume that we are in this case, and let $\pts = A \times B \subset \RR^{2}$.
By applying Theorem \ref{th:SharirZahl} with $s=3$, we obtain
\begin{align}
I(\pts,\curves_j) = O&\Big( |A|^{6/11+\eps'}|B|^{6/11+\eps'}|A-A|^{9/11}|B-B|^{9/11}k_{j}^{9/11} \nonumber \\
&+ |A|^{2/3}|B|^{2/3}k_{j}^{2/3}|A-A|^{2/3}|B-B|^{2/3}+k_{j}|A-A||B-B| + |A||B| \Big). \label{eq:IncShenUpper}
\end{align}

Since $k_j < |A||B|$, $|A-A|> |A|$, $|B-B|> |B|$, $|A-A| < |A|^2$, and $|B-B|<|B|^2$, we get that the bound on the right-hand side of \eqref{eq:IncShenUpper} is always dominated by its first term.
That is, we have
\begin{align} \label{eq:IncShenUpperSimplified}
I(\pts,\curves_j) = O\left( |A|^{6/11+\eps'}|B|^{6/11+\eps'}|A-A|^{9/11}|B-B|^{9/11}k_{j}^{9/11} \right).
\end{align}

For every $\delta \in D_{j}$, there are at least $j$ pairs $(a,b) \in A \times B$ that satisfy $f(a,b)=\delta$.
Fix such a pair $(a,b)$.
Then for every $a' \in A$ and $b' \in B$, there exist $\alpha\in A-A$ and $\beta \in B-B$ such that $f(a'+\alpha,b'+\beta)=f(a,b)=\delta$.
Thus, every $\delta \in D_j$ has at least $j|A||B|$ corresponding incidences in $\pts\times \curves_j$.
By summing this over every $\delta \in D_j$ with irreducible curves we get that $I(\pts,\curves_j) =\Omega( j|A||B|k_{j})$.
Combining this with \eqref{eq:IncShenUpperSimplified} gives
\[ j|A||B|k_{j} = O\left( |A|^{6/11+\eps'}|B|^{6/11+\eps'}|A-A|^{9/11}|B-B|^{9/11}k_{j}^{9/11} \right). \]
Rearranging leads to
\[ k_{j} = O\left( \frac{|A-A|^{9/2}|B-B|^{9/2} }{|A|^{5/2-11\eps'/2}|B|^{5/2-11\eps'/2}j^{11/2}}\right). \]

Recall that the above analysis holds only when $k_j \ge 2d$.
Let $J$ be the set of integers $j\ge 1$ that satisfy $1\le k_j <2d$.
Since every $\delta\in \Delta_j$ has at least $j$ corresponding distinct pairs in $A\times B$, we have the straightforward bound
\[ k_j = O(|A||B|/j). \]

We set $\eps = 11\eps'/9$ and $\gamma = \frac{|A-A||B-B|}{|A|^{7/9-\eps}|B|^{7/9-\eps}}$.
Combining \eqref{eq:EfUpper} with both of the above upper bounds for $k_j$ implies
\begin{align*}
E_{f}(A,B) &< \sum_{j=0}^{\log \mu(|A|+|B|)} 2^{2j+2} k_{2^j} \le 4\sum_{j=0}^{\log \gamma} 2^{2j} k_{2^j}  + 4\sum_{\log \gamma +1}^{\log \mu(|A|+|B|)} 2^{2j} k_{2^j}\\
&= O\left(\sum_{j=0}^{\log \gamma} |A||B|2^{j}  + \sum_{j=\log \gamma+1}^{\log \mu(|A|+|B|)} \frac{|A-A|^{9/2}|B-B|^{9/2} }{|A|^{5/2-11\eps'/2}|B|^{5/2-11\eps'/2}2^{7j/2}} + \sum_{j \in J}2^{2j} \cdot 2d\right)  \\
&= O\left(|A|^{2/9+\eps}|B|^{2/9+\eps}|A-A||B-B| +|A|^2+|B|^2\right).
\end{align*}

Combining this with \eqref{eq:EfLower} immediately implies the assertion of the theorem.
\end{proof}

We next recall the statement of Theorem \ref{th:GeneralStructure}.
\vspace{1mm}

\noindent {\bf Theorem \ref{th:GeneralStructure}.}
\emph{Let $A,B \subset \RR$ be finite sets and let $f \in \RR[x,y]$ be of degree at most $d$.
Let $\tau_A,\tau_B \in \RR[x]$ be of degree at most $d$, let $\deg \tau_B \ge 2$, and let $\tau_A^{-1}$ and $\tau_B^{-1}$ be respective one-sided inverses.
Assume that $f(\tau_A(x),\tau_B(y))$ is indecomposable, that $A \subset \tau_A(\RR)$, and that $B \subset \tau_B(\RR)$.
Then for any $\eps>0$ we have}
\[ |f(A,B)| = \Omega\left(\min\left\{\frac{|A|^{16/9-\eps} |B|^{16/9-\eps} }{|\tau^{-1}_A(A)-\tau^{-1}_A(A)||\tau^{-1}_B(B)-\tau^{-1}_B(B)|},|A|^2,|B|^2\right\}\right). \]
\vspace{1mm}

Unfortunately it is possible for $f(x,y)$ to be indecomposable and for $f(\tau_A(x),\tau_B(y))$ to be decomposable.
For example, $f(x,y)=xy$ is clearly indecomposable.
However, setting $\tau_A(x)=\tau_B(x)=x^2$ leads to $f(\tau_A(x),\tau_B(y)) = x^2y^2$ which is decomposable.
Characterizing exactly when this happens would lead to a stronger variant of Theorem \ref{th:GeneralStructure}.

\begin{proof}[Proof of Theorem \ref{th:GeneralStructure}.]
We use a variant of the proof of Theorem \ref{th:StructureElekRon}.
In particular, we define $\Delta, \mu, D, D_j, k_j, E_{f}(A,B)$, and $m_\delta$ as in the proof of Theorem \ref{th:StructureElekRon}.
Recall that every $\delta\in \Delta$ satisfies $m_\delta = O(|A|+|B|)$.
That is, there exists a constant $\mu$ such that $k_{\mu(|A|+|B|)} =0$.
As in \eqref{eq:EfUpper}, we have
\begin{align} \label{eq:EfUpperGen}
E_{f}(A,B) = \sum_{\delta\in \Delta}m_{\delta}^2 = \sum_{j=0}^{\log \mu(|A|+|B|) } \sum_{\delta \in \Delta \atop 2^{j} \leq m_\delta < 2^{j+1}}{m_{\delta}^2} < \sum_{j=0}^{\log \mu(|A|+|B|)} 2^{2j+2} k_{2^j}.
\end{align}
As in \eqref{eq:EfLower}, we have
\begin{align} \label{eq:EfLowerGen}
E_{f}(A,B) = \sum_{\delta\in \Delta}m_{\delta}^2 \ge \frac{(\sum_{\delta\in \Delta} m_\delta)^2}{D} = \frac{|A|^2|B|^2}{D}.
\end{align}

For brevity, we write $A_\tau=\tau_A^{-1}(A)-\tau_A^{-1}(A)$ and $B_\tau=\tau_B^{-1}(B)-\tau_B^{-1}(B)$.
For a fixed $j$, consider the set of curves
\begin{align*}
\curves_{j} =  \Big\{\vb(f(\tau_A(x+\alpha),\tau_B(y+\beta)) - \delta)\  :\  \alpha \in A_\tau,\ \beta \in B_\tau,\ \delta \in D_{j} \Big\}.
\end{align*}

Assume for contradiction that there exist $(\alpha,\beta,\delta), (\alpha',\beta',\delta') \in A_\tau \times B_\tau \times D_{j}$ and $r\in \RR$ such that
\[ f(\tau_A(x+\alpha),\tau_B(y+\beta)) -\delta = r \cdot \left(f(\tau_A(x+\alpha'),\tau_B(y+\beta')) - \delta'\right).\]
For the maximum degree terms on both sides to have the same coefficients, we must have $r=1$.
We replace $x$ with $x-\alpha'$ and $y$ with $y-\beta'$ (that is, we translate the plane by $-\alpha'$ in the $x$-direction and by $-\beta'$ in the $y$-direction). Setting $\alpha_{0}=\alpha-\alpha'$ and $\beta_{0}=\beta-\beta'$ leads to
\begin{equation} \label{eq:IdenticalCurvesPartB}
f(\tau_A(x+\alpha_{0}),\tau_B(y+\beta_{0})) = f(\tau_A(x),\tau_B(y)) - \delta' + \delta.
\end{equation}

One obvious solution to \eqref{eq:IdenticalCurvesPartB} is $\alpha_{0}=\beta_{0}=0$ and $\delta=\delta'$.
We assume for contradiction that this is not the case.
As explained in the proof of Theorem \ref{th:StructureElekRon}, this implies that $f(\tau_A(x),\tau_B(y))=h(x+cy)$ where $c=-\alpha_{0}/\beta_{0}$ and $h\in \RR[z]$.
Since $\deg \tau_B\ge 2$, we also have that $\deg h\ge 2$.
This in turn implies that $f(\tau_A(x),\tau_B(y))$ is decomposable, contradicting the assumption.

To recap, we proved that $\curves_{j}$ is a set of $|A_\tau| |B_\tau|k_j$ curves that are defined by distinct polynomials.
Moreover, no polynomial is a constant multiple of another.
While the polynomials are distinct, the curves of $\curves_{j}$ may still have common components.

By Theorem \ref{th:Stein}, there are at most $d^2$ elements $\delta \in \Delta$ for which $f(\tau_A(x),\tau_B(y)) -\delta$ is reducible.
For any $\alpha,\beta\in \RR$, since $\vb(f(\tau_A(x+\alpha),\tau_B(y+\beta)))$ is a translation of $\vb(f(\tau_A(x),\tau_B(y))-\delta)$, either both curves are reducible or both are irreducible.
Thus, at most $d|A-A||B-B|$ curves of $\curves_j$ are reducible.

If $k_j \ge 2d^2$, then after removing the reducible curves from $\curves_j$ we still have $|\curves_j| =\Theta(|A_\tau||B_\tau|k_j)$.
Assume that we are in this case, and let $\pts = \tau_A^{-1}(A) \times \tau_B^{-1}(B) \subset \RR^{2}$.
Note that $\tau_A^{-1}$ and $\tau_B^{-1}$ are injective, so $\pts = |A||B|$.
By repeating the above derivation of \eqref{eq:IncShenUpperSimplified}, we obtain
\begin{align} \label{eq:IncShenUpperSimplifiedPartB}
I(\pts,\curves_j) = O\left( |A|^{6/11+\eps'}|B|^{6/11+\eps'}|A_\tau|^{9/11}|B_\tau|^{9/11}k_{j}^{9/11} \right).
\end{align}

For every $\delta \in D_{j}$, there are at least $j$ pairs $(a,b) \in A \times B$ that satisfy $f(a,b)=\delta$.
For every $a' \in \tau_A^{-1}(A)$ and $b' \in \tau_B^{-1}(B)$, there exist $\alpha\in A_\tau$ and $\beta \in B_\tau$ such that $f(\tau_A(a'+\alpha),\tau_B(b'+\beta))=f(a,b)=\delta$.
Thus, every $\delta \in D_j$ whose curves are irreducible has at least $j|A||B|$ corresponding incidences in $\pts\times \curves_j$.
By summing this over every $\delta \in D_j$ with irreducible curves we get $I(\pts,\curves_j) =\Omega(j|A||B|k_j)$.
Combining this with \eqref{eq:IncShenUpperSimplifiedPartB} gives
\[ j|A||B|k_{j} = O\left( |A|^{6/11+\eps'}|B|^{6/11+\eps'}|A_\tau|^{9/11}|B_\tau|^{9/11}k_{j}^{9/11} \right). \]
Rearranging leads to
\[ k_{j} = O\left( \frac{|A_\tau|^{9/2}|B_\tau|^{9/2} }{|A|^{5/2-11\eps'/2}|B|^{5/2-11\eps'/2}j^{11/2}}\right). \]
We also recall the straightforward bound
\[ k_j = O(|A||B|/j). \]

Let $J$ be the set of integers $j\ge 1$ that satisfy $1\le k_j <2d^2$.
We set $\eps = 11\eps'/9$ and $\gamma = \frac{|A_\tau||B_\tau|}{|A|^{7/9-\eps}|B|^{7/9-\eps}}$.
Combining \eqref{eq:EfUpperGen} with both of the above bounds for $k_j$ implies
\begin{align*}
E_{f}(A,B) &< \sum_{j=0}^{\log \mu(|A|+|B|)} 2^{2j+2} k_{2^j} \le 4\sum_{j=0}^{\log \gamma} 2^{2j} k_{2^j}  + 4\sum_{\log \gamma +1}^{\log \mu(|A|+|B|)} 2^{2j} k_{2^j}.\\
&= O\left(\sum_{j=0}^{\log \gamma} |A||B|2^{j}  + \sum_{j=\log \gamma+1}^{\log \mu(|A|+|B|)} \frac{|A_\tau|^{9/2}|B_\tau|^{9/2} }{|A|^{5/2-11\eps'/2}|B|^{5/2-11\eps'/2}2^{7j/2}} + \sum_{j \in J}2^{2j} \cdot 2d^2\right)  \\
&= O\left(|A|^{2/9+\eps}|B|^{2/9+\eps}|A_\tau||B_\tau| +|A|^2+|B|^2\right).
\end{align*}
Combining this with \eqref{eq:EfLowerGen} completes the proof.
\end{proof}


\section{Bipartite distinct distances} \label{sec:Bipartite}

In this Section we prove Theorem \ref{th:BipartiteLineUnrest}.
We begin by recalling the statement of this theorem.
\vspace{1mm}

\noindent {\bf Theorem \ref{th:BipartiteLineUnrest}.}
\emph{Let $\pts_1$ be a set of $m$ points on a line $\ell$ in $\RR^2$ and let $\pts_2$ be set of $n$ points $\RR^2$.
Then}
\begin{equation*}
D(\pts_1,\pts_2) =
\begin{cases}
\Omega(m^{1/2}n^{1/2} \log^{-1/2}n), \hspace{8mm} \text{when }m=\Omega(n^{1/2}/\log^{1/3} n),\\
\Omega\left(n^{3/8}m^{3/4}\right), \hspace{22mm} \text{when }m=O(n^{1/2}/\log^{1/3} n) \text{ and }  m=\Omega(n^{3/10}),\\
\Omega\left(n^{1/2}m^{1/3}\right), \hspace{22mm} \text{when }m=O(n^{3/10}).
\end{cases}
\end{equation*}
\begin{proof}
For any $b\in \pts_2$ and distance $\delta$, at most two points of $\ell$ are at a distance of $\delta$ from $b$.
This implies that $D(\pts_1,\pts_2) \ge D(\pts_1,\{b\}) \ge m/2$ (for an arbitrary $b\in \pts_2$).
When $m= \Omega(n/\log n)$ this implies that $D(\pts_1,\pts_2) = \Omega(m^{1/2}n^{1/2} \log^{-1/2}n)$ and completes the proof.
We may thus assume that $m = O(n/\log n)$.

If at least half of the points of $\pts_2$ are contained in $\ell$, then for an arbitrary $a\in \pts_1$ we have $D(\pts_1,\pts_2) \ge D(\{a\},\pts_2\cap \ell) =\Theta\left(n\right)=\Omega\left(n^{1/2}m^{1/2}\right)$.
We may thus assume that at most half of the points of $\pts_2$ are in $\ell$.
Let $\pts_2' = \pts_2 \setminus \ell$, and note that $|\pts_2'|=\Theta(n)$.

We rotate the plane so that $\ell$ becomes the $x$-axis.
If at least half of the points of $\pts_2'$ have a negative $y$-coordinate then we reflect $\RR^2$ about the $x$-axis.
Let $\pts_2''$ be the set of points of $\pts_2'$ with a positive $y$-coordinate, and note that $|\pts_2''|=\Theta(n)$.
Since $D(\pts_1,\pts_2)\ge D(\pts_1,\pts_2'')$, it suffices to derive a lower bound on $D(\pts_1,\pts_2'')$.
Abusing notation, in the remainder of the proof we will refer to $\pts_2''$ as $\pts_2$ and refer to the size of this set as $n$.

The rest of the proof is based on double counting $E_{3}(\pts_1,\pts_2)$.
By \eqref{eq:dthEnergyLowerBipartite} we have
\begin{equation} \label{eq:BipartiteLineLower}
E_3(\pts_1,\pts_2) = \Omega\left(\frac{m^3n^3}{D(\pts_1,\pts_2)^{2}}\right).
\end{equation}

As before, for every $\delta \in \Delta$ we denote by $m_\delta$ the number of ordered pairs $(a,b)\in \pts^2$ such that $|ab|=\delta$.
Recall that for fixed $\delta \in \Delta$ and $b\in \pts_2$ at most two points $a\in \pts_1$ satisfy $|ab|=\delta$.
This implies that $m_\delta \le 2n$ for every $\delta \in \Delta$.
Let $\Delta_j$ be the set of distances $\delta \in \Delta$ that satisfy $m_\delta \ge j$, and set $k_j = |\Delta_j|$.
A dyadic decomposition argument gives
\begin{align}
E_{3}(\pts_1,\pts_2) = \sum_{j=0}^{\log(2n)} \sum_{\delta\in \Delta \atop 2^{j} \leq m_\delta < 2^{j+1}}{m_\delta^{3}} &< \sum_{j=0}^{\log(2n)} \sum_{\delta\in \Delta \atop 2^{j} \leq m_\delta < 2^{j+1}}(2^{j+1})^3 \le 8 \sum_{j=0}^{\log(2n)} 2^{3j} k_{2^j}. \label{eq:BipartiteThirdEnergyUpper}
\end{align}

\parag{Studying rich distances.}
Fix a positive integer $j$.
With \eqref{eq:BipartiteThirdEnergyUpper} in mind, we now study how large $j^3k_j$ can be.
Consider the set of circles
\[ \curves_j = \left\{\vb((x-a_x)^2+ y^2 - \delta^2) \subset \RR^2 :\ \delta \in \Delta_j, (a_x,0) \in \pts_1 \right\}. \]

Note that $\curves_j$ is a set of $mk_j$ distinct circles.
Since the centers of the circles are on the $x$-axis, two such circles intersect in at most one point with a positive $y$ coordinate.
Thus, the incidence graph of $\pts_2\times \curves_j$ contains no $K_{2,2}$.
By Theorem \ref{th:PachSharir}
\begin{equation} \label{eq:BipartiteIncUpper}
I(\pts_2,\curves_j) = O\left(m^{2/3}n^{2/3} k_j^{2/3} + n + mk_j\right).
\end{equation}

We divide the analysis into cases according to the term that dominates the bound of \eqref{eq:BipartiteIncUpper}.
If $m^{2/3}n^{2/3}k_j^{2/3} = O(mk_{j})$ then $D(\pts_1,\pts_2) \ge k_j =\Omega(n^{2}/m) = \Omega(n^{1/2}m^{1/2})$.
This completes the proof, so we may ignore this case.
If $m^{2/3}n^{2/3}k_j^{2/3} = O(n)$ then $k_j =O\left(n^{1/2}/m\right)$.
This in turn implies
\[ j^3k_j =O\left(j^3 n^{1/2}/m\right). \]

Finally, consider the case where $m^{2/3}n^{2/3} k_j^{2/3}$ dominates the bound of \eqref{eq:BipartiteIncUpper}.
For a distance $\delta \in \Delta_j$, every representation of $\delta$ as a distance between $a\in \pts_1$ and $b\in \pts_2$ corresponds to an incidence between $b$ and the circle defined by $a$ and $\delta$.
Since each of the $k_j$ distances of $\Delta_j$ has at least $j$ such representations, we get that $I(\pts,\curves_j) \ge j k_j$.
Combining this with \eqref{eq:BipartiteIncUpper} gives $j k_j = O(m^{2/3}n^{2/3} k_j^{2/3})$.
Tidying up leads to
\[ j^3k_j = O\left(m^2n^2\right).\]

By combining the two bounds for $j^3k_j$ with \eqref{eq:BipartiteThirdEnergyUpper}, we obtain
\begin{align}
E_3(\pts_1,\pts_2) < 8 \sum_{j=0}^{\log(2n)} 2^{3j} k_{2^j} &= O\left(\sum_{j=0}^{\log(2n)} \left(m^2n^2 + \frac{2^{3j} n^{1/2}}{m}\right)\right) \nonumber \\
&\hspace{45mm}= O\left(m^2n^2\log n + \frac{n^{7/2}}{m}\right). \label{eq:bipartiteE3UpperGeneral}
\end{align}
When $m=\Omega(n^{1/2}/\log^{1/3} n)$, we get that $E_3(\pts_1,\pts_2) = O\left(m^{2}n^{2} \log n\right)$.
Combining this with \eqref{eq:BipartiteLineLower} implies the asserted bound
\[ D(\pts_1,\pts_2) =\Omega\left(n^{1/2}m^{1/2} \log^{-1/2} n\right). \]

\parag{The case of small $m$.}
We now consider the case where $m=O(n^{1/2}/\log^{1/3} n)$.
We present two different arguments that lead to the two remaining bounds in the statement of the theorem.
First, assume that there exists $\delta\in \Delta$ such that $m_\delta \ge n^{1/2} m^{4/3}$.
Let $\circs$ be the set of circles of radius $\delta$ that are centered at points of $\pts_1$, and note that $I(\pts_2,\circs)\ge n^{1/2}m^{4/3}$.
By the pigeonhole principle there exists a circle $\gamma\in \circs$ that is incident to at least $n^{1/2}m^{1/3}$ points of $\pts_2$.
 For an arbitrary $a\in \pts_1$ that is not the center of $\gamma$, we get that $D(\pts_1,\pts_2) \ge D(\{a\}, \pts_2\cap \gamma) \ge n^{1/2}m^{1/3}/2$.

Next, assume that every $\delta\in \Delta$ satisfies $m_\delta < n^{1/2} m^{4/3}$.
Since every pair $\delta \in \Delta_j$ corresponds to at least $j$ distinct ordered pairs of $\pts_1\times \pts_2$, we have the straightforward bound $k_j \le mn/j$.
We use this bound for $j = O(m^{1/2}n^{1/2})$, and for larger values of $j$ we use the bound $k_j = O(n^{1/2}m^{-1}+m^2n^2j^{-3})$ derived in the previous part of this proof.
Repeating the argument in \eqref{eq:BipartiteThirdEnergyUpper} for $E_{2}(\pts_1,\pts_2)$, we get that
\begin{align*}
E_{2}(\pts_1,\pts_2) &< 4 \sum_{j=0}^{\log n^{1/2} m^{4/3}} 2^{2j} k_{2^j} = 4 \sum_{j=0}^{\log \sqrt{mn}} 2^{2j} k_{2^j} + 4 \sum_{j=\log\sqrt{mn}}^{\log n^{1/2} m^{4/3}} 2^{2j} k_{2^j} \\
&= O\left( \sum_{j=0}^{\log \sqrt{mn}} mn2^{j} +  \sum_{j=\log \sqrt{mn}}^{\log n^{1/2} m^{4/3}} \left(2^{2j} n^{1/2}m^{-1} + m^2n^22^{-j}\right)\right) = O\left(n^{3/2}m^{5/3}\right).
\end{align*}

By \eqref{eq:dthEnergyLowerBipartite} we have
\begin{equation} \label{eq:BipartiteLineLowerE2}
E_2(\pts_1,\pts_2) = \Omega\left(\frac{m^2n^2}{D(\pts_1,\pts_2)}\right).
\end{equation}
Combining the two above bounds for $E_2(\pts_1,\pts_2)$ gives $D(\pts_1,\pts_2) =\Omega\left(n^{1/2}m^{1/3}\right)$.
Thus, in either case we have that
\[ D(\pts_1,\pts_2) =\Omega\left(n^{1/2}m^{1/3}\right). \]

For our final bound, assume that $m=\Omega(n^{3/10})$ (and $m=O(n^{1/2}/\log^{1/3}n)$).
If there exists $\delta\in \Delta$ such that $m_\delta \ge n^{9/16} m^{9/8}$,
repeating the above argument involving the set of circles $\circs$ gives a circle $\gamma$ that contains $\Omega(n^{9/16} m^{1/8})$ points of $\pts_2$.
By Theorem \ref{th:PachdeZeeuw}, we have
\begin{align*}
D(\pts_1,\pts_2) \ge D(\pts_1,\pts_2\cap \gamma) &= \Omega\left(\min\left\{|\pts_2 \cap \gamma|^{2/3}m^{2/3},|\pts_2 \cap \gamma|^2, m^2\right\}\right) \\
&= \Omega\left(\min\left\{n^{3/8} m^{3/4},n^{9/8} m^{1/4}, m^2\right\}\right) = \Omega\left(n^{3/8} m^{3/4}\right).
\end{align*}

On the other hand, if every $\delta\in \Delta$ satisfies $m_\delta < n^{9/16} m^{9/8}$ then
\begin{align*}
E_{2}(\pts_1&,\pts_2) < 4\sum_{j=1}^{\log n^{9/16} m^{9/8}} 2^{2j} k_{2^j} < 4 \sum_{j=1}^{\log \sqrt{mn}} 2^{2j} k_{2^j} + 4 \sum_{j=\log\sqrt{mn}}^{\log n^{9/16} m^{9/8}} 2^{2j} k_{2^j} \\
&= O\left( \sum_{j=1}^{\log \sqrt{mn}} mn2^{j} +  \sum_{j=\log\sqrt{mn}}^{\log n^{9/16} m^{9/8}} \left(2^{2j} n^{1/2}m^{-1} + m^2n^22^{-j}\right)\right) = O\left(n^{13/8}m^{5/4}\right).
\end{align*}
Combining this with \eqref{eq:BipartiteLineLowerE2} implies $D(\pts_1,\pts_2) =\Omega\left(n^{3/8}m^{3/4}\right)$.
Thus, in either case we get
\[ D(\pts_1,\pts_2) =\Omega\left(n^{3/8}m^{3/4}\right). \]
\end{proof}

\section{Subsets with few repeating distances} \label{sec:Subset}

In this section we prove Theorem \ref{th:Subset} and another related result.
We begin by recalling the statement of this theorem.
\vspace{1mm}

\noindent {\bf Theorem \ref{th:BipartiteLineUnrest}.}
\emph{Let $\pts \subset \RR^2$ be a set of $n$ points.
Then there exists a subset $\pts' \subset \pts$ of size $\Omega(n^{22/63} \log^{-13/63}n)$ such that no distance is spanned by more than four pairs of points of $\pts'$.
Similarly, there exists a subset $\pts' \subset \pts$ of size $\Omega(n^{12/35} \log^{-9/35} n)$ such that no distance is spanned by more than two pairs of points of $\pts'$.}
\begin{proof}
Let $\Delta_j$ be the set of distances that are spanned by at least $j$ pairs of points of $\pts$.
Set $k_j = |\Delta_j|$.
We begin the proof by studying how large $k_j$ can be.

For a fixed $j$, consider the set of circles
\[ \curves_{j} = \left\{ \vb\left((x-a_{x})^2+(y-a_{y})^2- \delta^2\right)\  :\  \delta \in \Delta_j, \ a \in \pts \right\}. \]
Note that $\curves_{j}$ is a set of $n k_j$ distinct circles.
By Theorem \ref{th:IncSharirZahl} and the trivial bound $k_{j} =O(n^2)$, we have
\begin{align}
I(\pts,\curves_j) &= O\left(n^{6/11}(nk_j)^{9/11}\log^{2/11}(nk_j)+n^{2/3}(nk_j)^{2/3}+n+nk_j\right) \nonumber \\ &\hspace{60mm}=O\left(n^{15/11}k_j^{9/11}\log^{2/11}n\right). \label{eq:SubsetIncUpper}
\end{align}

For every distance $\delta \in \Delta_j$ there are at least $j$ pairs of points in $\pts$ that span $\delta$.
Each such pair corresponds to two incidences in $\pts\times \curves_j$, so
\[ I(\pts,\curves_{j}) \geq 2j k_{j}. \]
Combining this with \eqref{eq:SubsetIncUpper} yields
\begin{equation} \label{eq:SubsetRich}
 k_j = O\left(\frac{n^{15/2}\log n}{j^{11/2}}\right).
\end{equation}

Let $\Delta$ be the set of distances spanned by pairs of points of $\pts$.
A dyadic pigeonholing argument gives
\begin{align*}
\sum_{\delta\in \Delta}m_\delta^{11/2} = \sum_{j=0}^{2\log n} \sum_{\delta\in \Delta \atop 2^j \le m_\delta < 2^{j+1}}m_\delta^{11/2} < \sum_{j=0}^{2\log n} k_{2^j} (2^{j+1})^{11/2} &= O\left(\sum_{j=0}^{2\log n}n^{15/2}\log n\right) \\
&= O\left(n^{15/2}\log^2 n\right).
\end{align*}

Recall from \eqref{eq:GuthKatzE2} that $\sum_{\delta\in \Delta}m_\delta^{2} = O(n^3\log n)$.
By H\"older's inequality
\begin{align} E_5^*(\pts) \le E_5(\pts) = \sum_{\delta\in \Delta}&m_\delta^5 = \sum_{\delta\in \Delta}m_\delta^{33/7}m_\delta^{2/7} \le \left(\sum_{\delta\in \Delta}m_\delta^{11/2}\right)^{6/7}\left(\sum_{\delta\in \Delta} m_\delta^2\right)^{1/7} \nonumber \\
&= O\left(\left(n^{15/2}\log^2 n\right)^{6/7}\left(n^3\log n\right)^{1/7}\right) = O\left(n^{48/7}\log^{13/7} n\right). \label{eq:SubsetE5*Upper}
\end{align}

\parag{A probabilistic argument.}
For $0< p<1$ that will be determined below, let $\pts''$ be a set that is obtained by taking every point of $\pts$ with probability $p$.
Note that\footnote{We denote by $\EE[X]$ the expectation of the random variable $X$, to distinguish it from the energy notation $E(X)$.} $\EE[|\pts''|] = pn$, that $\EE[E_5^*(\pts'')] = p^{10} \cdot E_5^*(\pts)$, and that the expected number of isosceles triangles is $\EE[t(\pts'')] = p^3 \cdot t(\pts)$.
By linearity of expectation, Theorem \ref{th:Isosceles}, and \eqref{eq:SubsetE5*Upper}, we have that
\begin{align*}
\EE\Bigg[|\pts''|- E_5^*(\pts'') - t(\pts'') \Bigg] &= pn - p^{10} \cdot E_5^*(\pts) - p^3 \cdot t(\pts) \\
&\ge pn - cp^{10} n^{48/7}\log^{13/7} n - cp^3 n^{2.137},
\end{align*}
for a sufficiently large constant $c$.

To asymptotically maximize the above expectation, we set $p= \left(2cn^{41/7}\log^{13/7} n\right)^{-1/9}$.
This implies
\begin{align*}
\EE\Bigg[|\pts''|- E_5^*(\pts'') - t(\pts'') \Bigg] \ge \frac{n^{22/63}}{\left(2c\log^{13/7} n\right)^{1/9}} -  & \frac{n^{22/63}}{2\left(2c\log^{13/7} n\right)^{1/9}} - \frac{(c)^{2/3} n^{0.184}}{\left(2\log^{13/7} n\right)^{1/3}} \\
&\hspace{26 mm}= \Omega\left(n^{22/63}\log^{-13/63} n\right).
\end{align*}

By the above, there exists $\pts''\subset \pts$ such that $|\pts''|- E_5^*(\pts'') - t(\pts'') = \Omega(n^{22/63}\log^{-13/63} n)$.
We create $\pts' \subset \pts''$ by taking $\pts''$ and arbitrarily removing one vertex from every isosceles triangle that is spanned by $\pts''$ and from every 10-tuple that contributes to $E_5^*(\pts'')$.
By the above, $|\pts'|= \Omega(n^{22/63}\log^{-13/63} n)$.
Note that no distance is spanned by more than four pairs of points of $\pts'$.
This completes the proof of the first statement of the theorem.

\parag{A subset with no distance repeating more than twice.}
We now prove the second statement of the theorem.
The proof is almost identical to the preceding one, except that $E_5^*(\pts)$ is replaced with $E_3^*(\pts)$.
By H\"older's inequality
\begin{align} E_3^*(\pts) \le E_3(\pts) = \sum_{\delta\in \Delta}&m_\delta^3 = \sum_{\delta\in \Delta}m_\delta^{11/7}m_\delta^{10/7} \le \left(\sum_{\delta\in \Delta}m_\delta^{11/2}\right)^{2/7}\left(\sum_{\delta\in \Delta} m_\delta^2\right)^{5/7} \nonumber \\
&= O\left(\left(n^{15/2}\log^2 n\right)^{2/7}\left(n^3\log n\right)^{5/7}\right) = O\left(n^{30/7}\log^{9/7} n\right). \label{eq:SubsetE3*Upper}
\end{align}

We randomly generate $\pts''$ as above.
By linearity of expectation, Theorem \ref{th:Isosceles}, and \eqref{eq:SubsetE3*Upper}, we have
\begin{align*}
\EE\Bigg[|\pts''|- E_3^*(\pts'') - t(\pts'') \Bigg] &= pn - p^{6} \cdot E_3^*(\pts) - p^3 \cdot t(\pts) \\
&\ge pn - cp^{6} n^{30/7}\log^{9/7} n - cp^3 n^{2.137},
\end{align*}
for a sufficiently large constant $c$.

To asymptotically maximize the above expectation, we set $p= \left(2cn^{23/7}\log^{9/7} n\right)^{-1/5}$.
This implies
\begin{align*}
\EE\Bigg[|\pts''|- E_5^*(\pts'') - t(\pts'') \Bigg] \ge \frac{n^{12/35}}{\left(2c\log^{9/7} n\right)^{1/5}} -  & \frac{n^{12/35}}{2\left(2c\log^{9/7} n\right)^{1/5}} - \frac{(c)^{2/5} n^{0.166}}{\left(2\log^{9/7} n\right)^{3/5}} \\
&\hspace{29 mm}= \Omega\left(n^{12/35}\log^{-9/35} n\right).
\end{align*}

The final part of the argument is identical to the one in the previous case.
\end{proof}

If Conjecture \ref{co:HigherDistEnergy} holds, then the proof of Theorem \ref{th:Subset} would imply significantly stronger results.
For example, if $E_3(\pts) = O(n^{4+\eps})$ then we would get that there exists a subset $\pts' \subset \pts$ of size $\Omega(n^{2/5-\eps})$ such that no distance is spanned more than twice by pairs of points of $\pts'$.

There are many variants of Conjecture \ref{co:subset}, and we can derive similar style results for most of those by using higher distance energies.
We now present one such variant by Raz \cite{Raz16}.

\begin{theorem}\label{th:RazSubset}
Let $\gamma\subset \RR^d$ be an irreducible algebraic curve of degree $k$ and let $\pts$ be a set of $n$ points on $\gamma$.
Then there exists a subset $\pts'\subset \pts$ of size $\Omega(n^{4/9})$ that does not span any distance more than once.
\end{theorem}

Combining Theorem \ref{th:RazSubset} with a work of Conlon et al.\ \cite{CFGHUZ15} leads to a family of results involving subsets that do not span simplices with repeating volumes.
As an upper bound for Theorem \ref{th:RazSubset}, when taking a set of $n$ equally spaced points on a line we get that every subset of size $\Omega(\sqrt{n})$ contains a repeating distance (this is the \emph{Sidon set} problem).
By using higher distance energies, we obtain the following variant.

\begin{theorem} \label{th:SubsetOnCurve}
Let $\gamma\subset \RR^d$ be an irreducible algebraic curve of degree $k$ and let $\pts$ be a set of $n$ points on $\gamma$.
Then for every integer $m\ge 2$ there exists a subset $\pts'\subset \pts$ such that $|\pts'|= \Omega\left(n^{\frac{3m-2}{6m-3}}\right)$ and no distance is spanned by more than $m-1$ pairs of points of $\pts'$.
\end{theorem}

Theorem \ref{th:SubsetOnCurve} implies that for every $\eps>0$ there exists $\pts'\subset \pts$ such that $|\pts'|= \Omega\left(n^{0.5-\eps}\right)$ and every distance is spanned by $O(1)$ pairs of points of $\pts'$.
\begin{proof}
By the proof of Theorem 4.1 in \cite{Raz16}, either there exists a subset of $\Theta(n^{1/2})$ points of $\pts$ that do not span any distance more than once, or there exists a subset $\tts\subset \pts$ such that $|\tts|=\Theta(n)$ and $E_2(\tts)=O(|\tts|^{8/3})$.
In the former case we are done, so assume that the latter case holds.
Set $n_\tts = |\tts|=\Theta(n)$.

Consider a point $p\in \tts$ and a distance $\delta$.
The points of $\RR^d$ that are at a distance of $\delta$ from $p$ form a hypersphere $S$ centered at $p$.
Note that $\gamma \not\subset S$, since $p\in \gamma\setminus S$.
Since $\gamma$ and $S$ are irreducible varieties with no common components, we get that $\gamma \cap S$ is a finite point set.
Theorem \ref{th:Milnor} implies that $|\gamma \cap S|= O(1)$.
That is, $p$ is at a distance of $\delta$ from $O(1)$ points of $\tts$.
This in turn implies that $t(\tts)=O(n_\tts^2)$.

For an integer $m\ge 2$, we consider the size of $E_m(\tts)$.
Since $E_2(\tts)=O(n_\tts^{8/3})$, there are $O(n_\tts^{8/3})$ quadruples $(a_1,a_2,b_1,b_2)\in \tts$ such that $|a_1b_1|=|a_2b_2|$.
Fix such a quadruple $(a_1,a_2,b_1,b_2)$, together with additional points $a_3,\ldots, a_m\in \tts$.
By the previous paragraph, there are $O(1)$ choices of $b_3,\ldots,b_m\in \tts$ such that $|a_1b_1|=|a_2b_2| = \cdots = |a_mb_m|$.
This implies that
\begin{equation} \label{eq:CurveSubsetEnergym}
E_m^*(\tts) \le E_m(\tts) = O\left(n_\tts^{m+2/3}\right).
\end{equation}

\parag{A probabilistic argument.}
For $0< p<1$ that will be determined below, let $\pts''$ be a set that is obtained by taking every point of $\tts$ with probability $p$.
Note that $\EE[|\tts|] = pn_\tts$, that $\EE[E_m^*(\tts)] = p^{2m} \cdot E_m^*(\tts)$, and that $\EE[t(\pts'')] = p^3 \cdot t(\tts)$.
By linearity of expectation, the aforementioned bound $t(\tts)=O(n_\tts^2)$, and \eqref{eq:CurveSubsetEnergym}, we have that
\begin{align*}
\EE\Bigg[|\pts''|- E_m^*(\pts'') - t(\pts'') \Bigg] &= pn_\tts - p^{2m} \cdot E_m^*(\tts) - p^3 \cdot t(\tts) \\
&\ge pn_\tts - cp^{2m} n_\tts^{m+2/3} - cp^3 n_\tts^2,
\end{align*}
for a sufficiently large constant $c$ (which may depend on $d,k$, and $m$).

To asymptotically maximize the above expectation, we set $p= \left(2cn_\tts^{m-1/3}\right)^{-1/(2m-1)}$.
This implies
\begin{align*}
\EE\Bigg[|\pts''|- E_m^*(\pts'') - t(\pts'') \Bigg] \ge &\frac{n_\tts^{\frac{3m-2}{6m-3}}}{\left(2c\right)^{\frac{1}{2m-1}}} -  \frac{n_\tts^{\frac{3m-2}{6m-3}}}{2\left(2c\right)^{\frac{1}{2m-1}}} - \frac{c^{\frac{2m-4}{2m-1}} n_\tts^{\frac{m-1}{2m-1}}}{2^{\frac{3}{2m-1}}} \\
&\hspace{26 mm}= \Omega\left(n_\tts^{\frac{3m-2}{6m-3}} \right) = \Omega\left(n^{\frac{3m-2}{6m-3}} \right).
\end{align*}

By the above, there exists $\pts''\subset \tts$ such that $|\pts''|- E_m^*(\pts'') - t(\pts'') = \Omega\left(n^{\frac{3m-2}{6m-3}}\right)$.
We create $\pts' \subset \pts''$ by taking $\pts''$ and arbitrarily removing one vertex from every isosceles triangle that is spanned by $\pts''$ and from every $2m$-tuple that contributes to $E_m^*(\pts'')$.
By the above, $|\pts'|= \Omega\left(n^{\frac{3m-2}{6m-3}}\right)$.
No distance is spanned by more than $m-1$ pairs of points of $\pts'$.
\end{proof}


\end{document}